\documentclass[10pt,psamsfonts]{amsart}
\usepackage{amsmath}
\usepackage{amsthm}
\usepackage{amssymb}
\usepackage{amscd}
\usepackage{amsfonts}
\usepackage{amsbsy}
\usepackage{graphicx}
\usepackage[dvips]{psfrag}
\usepackage{array}
\usepackage{color}
\usepackage{epsfig}
\usepackage{url}
\usepackage{verbatim}
\usepackage{overpic}

\newcolumntype{L}{>{\displaystyle}l}
\newcolumntype{C}{>{\displaystyle}c}
\newcolumntype{R}{>{\displaystyle}r}

\newcommand{\R}{\ensuremath{\mathbb{R}}}

\newcommand{\Z}{\ensuremath{\mathbb{Z}}}

\newcommand{\CO}{\ensuremath{\mathcal{O}}}

\newcommand{\sgn}{\mathrm{sign}}

\def\e{\varepsilon}

\newtheorem {theorem} {Theorem}

\newtheorem {proposition}[theorem]{Proposition}
\newtheorem {corollary}[theorem] {Corollary}
\newtheorem {lemma}[theorem]{Lemma}

\newtheorem {remark} {Remark}

\newtheorem {mtheorem} {Theorem}

\begin{document}
\renewcommand{\arraystretch}{1.5}

\title[Simultaneous sliding  and crossing limit cycles]
{Simultaneous occurrence of sliding  and crossing limit cycles in
piecewise linear planar vector fields}

\author[J. L. Cardoso, J. Llibre, D. D. Novaes and D. J. Tonon]
{Jo\~ao L. Cardoso$^1$, Jaume Llibre$^2$, Douglas D. Novaes$^3$  and
Durval J. Tonon$^1$}

\address{$^1$ Institute of Mathematics and Statistics of Federal
University of Goi\'{a}s, Avenida Esperan\c{c}a s/n, Campus Samambaia,
74690-900, Goi\^{a}nia, Goi\'{a}s, Brazil}
\email{joao.lopes@ifg.edu.br} \email{djtonon@ufg.br}

\address{$^2$ Departament de Matematiques, Universitat Aut\`{o}noma
de Barcelona, 08193 Bellaterra, Barcelona, Catalonia, Spain}
\email{jllibre@mat.uab.cat}

\address{$^3$ Departamento de Matematica, Universidade
Estadual de Campinas, Rua S\'{e}rgio Buarque de Holanda, 651, Cidade
Universit\'{a}ria, 13083-859, Campinas, S\~{a}o Paulo, Brazil}
\email{ddnovaes@ime.unicamp.br}



\maketitle

\begin{abstract}
In the present study we consider planar piecewise linear vector fields with two zones separated by the straight line $x=0$. Our goal is to study the existence of simultaneous crossing and sliding limit cycles for such a class of vector fields. First, we provide a canonical form for these systems assuming that each linear system has center, a real one for $y<0$ and a virtual one for $y>0$, and such that the real center is a global center. Then, working with a first order piecewise linear perturbation we obtain piecewise linear differential systems with three crossing limit cycles. Second, we see that a sliding cycle can be detected after a second order piecewise linear perturbation. Finally, imposing the existence of a sliding limit cycle we prove that only one additional crossing limit cycle can appear. Furthermore, we also characterize the stability of the higher amplitude limit cycle and of the infinity. The main techniques used in our proofs are the Melnikov method, the Extended Chebyshev systems with positive accuracy, and the Bendixson transformation.
\end{abstract}


\section{Introduction and statement of the main results}

For a given differential system a limit cycle is a periodic orbit
isolated in the set of all periodic orbits of the system. One of the
main problems of the qualitative theory of planar differential
systems is determining the existence of limit cycles. A {\it center}
is a singular point $p$  that possesses a neighborhood $U$ such that
$U\setminus\{p\}$ is filled by periodic solutions. A classical way to
produce and study limit cycles is by perturbing the periodic
solutions of a center. This problem has been studied intensively for
continuous planar differential systems, see for instance, \cite{CL}
and the references therein.

In this paper, we are concerned in limit cycles bifurcating from a center of  discontinuous piecewise linear
differential systems with two zones separated by the straight line $x
= 0,$ when the center is perturbed inside the class of all
discontinuous piecewise linear differential systems with two zones
separated by $x = 0$. 

The study of the piecewise linear differential systems goes back to
Andronov and co-authors \cite{AVK}. In the present days these systems
continue receiving a considerable attention, mainly due to their
applications. Indeed, such systems are widely used to model many real
processes and different modern devices, see for instance the book
\cite{dBBCKN} and the references therein.

The case of {\it continuous piecewise linear systems} with two
regions separated by a straight line is the simplest possible
configuration of piecewise linear systems. We note that even in this
simple case, establishing the existence of at most one limit cycles
was a difficult task, see \cite{FPRT}, and \cite{LOP} for a shorter
proof. There are two reason for that misleading simplicity of
piecewise linear systems: first, whereas it is easy to compute the
solutions in any linear region, the time that each orbit requires to
pass from one linear region to the other is usually not computable;
second, the increasing number of parameters.

{\it Discontinuous piecewise linear systems}  with two regions
separated by a straight line have received a lot of attention during
the last years, see for instance \cite{FPRTb,HZ, HY, LNT, LNTb, LP,N}
among other papers. In \cite{HZ},  the authors conjectured that
piecewise linear systems with two regions separated by a straight
line could have at most  two crossing limit cycles. Later on in \cite{HY}, the
authors provided numerical evidence on the existence of three crossing limit
cycles, which was analytically proved in \cite{LP}. Sufficient condition on piecewise linear system implying the existence of at most 3 crossing limit cycles can be found in \cite{FPRTb,LNTb,N}.  As far as we know, there are no examples of piecewise linear vector fields separated by a straight line with more than 3 crossing limit cycles. In fact, although there is no proof, it is common sense that 3 is very likely the upper bound in this case. It is worthwhile to mention that the shape of the discontinuity set plays an important role in the number of crossing limit cycles. Indeed, if the discontinuity set is not a straight line, then one may find an arbitrary number of crossing limit cycles (see \cite{NP}).

Notice that all the above discussion about the number of limit cycles for piecewise linear systems is only concerned with crossing limit cycles. Thus, it is natural to inquire what happens when one also considers  sliding limit cycles. In this way, some natural questions arise on the coexistence of sliding and crossing limit cycles: Is it possible to have more than 3 limit cycles when sliding limit cycles are allowed? In other words, is it possible to find examples with 3 crossing limit cycles and one additional sliding limit cycle? 

Accordingly, in the present paper, we aim to investigate the simultaneous existence of sliding limit cycles and how they change the upper bound for the number of limit cycles for piecewise linear vector fields with two zones separated by the straight line $x=0$. Firstly, we provide a canonical form for a class of these systems assuming that each
linear system has a center, a real one for $y<0$ and a virtual one for
$y>0$, such that the real center is a global center. That is, the period annulus of the real center is the whole plane but the singularity. Then, proceeding
with a first order piecewise linear  perturbation, we study the
bifurcation of crossing limit cycles. In short, it is shown that any
configuration of $1$, $2$ or $3$ crossing limit cycles can bifurcate from the
periodic solutions of the unperturbed system.  We also study the stability of the higher amplitude crossing limit cycle and of the infinity. Secondly, we see that a sliding cycle can be detected after
a second order piecewise linear perturbation. Finally, imposing the
existence of a sliding limit cycle we prove that only one additional
crossing limit cycle can appear. Therefore, we conclude that at a second order analysis and for this particular class of piecewise linear vector fields, the sliding limit cycles do not increase the upper bound for the number of limit cycles and, additionally, a sliding limit cycle coexists with at most one crossing limit cycle.

\subsection{Setting the problem}

Consider the following planar piecewise linear differential systems
with two zones separated by the straight line $\Sigma=\{(x,y);x=0\}$
\begin{equation}\label{us1}
\left(\begin{array}{c}
\dot x\\
\dot y
\end{array}\right)
=W_{\e}(x,y)=\left\{\begin{array}{lc} M^+_{\e}\left(\begin{array}{c}
x\\
y
\end{array}\right) + \left(\begin{array}{c}
u_{1,\e}^+\\
u_{2,\e}^+
\end{array}\right) & \textrm{if} \quad x\geq 0, \vspace{0.2cm}\\
M^-_{\e}\left(\begin{array}{c}
x\\
y
\end{array}\right)+\left(\begin{array}{c}
u_{1,\e}^-\\
u_{2,\e}^-
\end{array}\right) & \textrm{if} \quad x\leq 0,
\end{array}\right.
\end{equation}
where  $M^{\pm}_{\e}$ is one-parameter family of $2\times 2$ real
matrices and $u_{1,\e}^{\pm},u_{2,\e}^{\pm}\in \R$. Here, the dot
denotes derivative with respect to the time variable $t$. We define
$\Sigma^+=\{(x,y); x >0\}$ and $\Sigma^-=\{(x,y); x <0\}$ and denote
the piecewise linear vector field associated to system \eqref{us1} by
$W_{\e}=(W_{\e}^+,W_{\e}^-)$, where $W_{\e}^+$ and $W_{\e}^-$ are
defined in $\Sigma^+$ and $\Sigma^-$, respectively.

Let $p^{\pm}$ be the singular points of $W_0^{\pm}$. We say that a
singular point $p^{\pm}$ of $W_0^{\pm}$ is {\it real} (respectively
{\it virtual}) if $p^{+}\in\left\lbrace (x,y); x\geq 0\right\rbrace $
or $p^{-}\in\left\lbrace (x,y); x\leq 0\right\rbrace$ (respectively
if $p^{+}\in\left\lbrace (x,y); x\leq 0\right\rbrace $ or
$p^{-}\in\left\lbrace (x,y); x\geq 0\right\rbrace$). Assume that
system \eqref{us1} satisfies
\begin{itemize}
\item[$(H_1)$] $p^-$ is a center for the system $W_0^-$ and $p^-\in
\Sigma^-$, i.e, $p^-$ is a real singular point.

\smallskip

\item[$(H_2)$] $p^+$ is a center for the system $W_0^+$ and $p^+
\in \Sigma^-$,  i.e, $p^+$ is a virtual singular point.

\smallskip

\item[$(H_3)$] $p^-$ is global center for the system $W_0$,
see Figure \ref{GraficoCentroDescontinuo}.
\end{itemize}

\begin{figure}[h]
\begin{center}
\begin{overpic}[width=5cm]{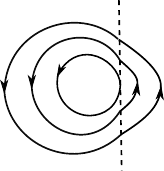}
\put(42,43){$p^{-}$}\put(49,48){$\bullet$}\put(60,95){$\Sigma$}
\end{overpic}
\end{center}
\caption{The global center for the system
$W_0$.}\label{GraficoCentroDescontinuo}
\end{figure}

\subsection{Canonical Form}

The next result provides a canonical form for the piecewise linear
vector fields \eqref{us1} having a center and it is proved in Section
\ref{s2}. Fixing $\e=0$, we denote
\[
M^{\pm}_0=\left(\begin{array}{cc}
m_{11}^{\pm}&m_{12}^{\pm}\\
m_{21}^{\pm}&-m_{11}^{\pm}\end{array}\right).
\]

\begin{proposition}\label{p1}
Consider the piecewise linear vector field $W_{\e}(x,y)$ given in
\eqref{us1} and assume that $W_\e(x,y)$ satisfies hypotheses $(H_1),
(H_2)$ and $(H_3).$  Then,   there exists a change of coordinates
$(t,x,y)\mapsto(\tilde t,x,\tilde y)$ where the piecewise linear
vector field \eqref{us1} can be written as $(x',\tilde y')^T =
Z_{\e}(x,\tilde y)$, where
\begin{equation*}\label{us}
Z_0(x,\tilde y) =\left\{\begin{array}{LC}Z_0^+(x, \tilde y)=
A^+\left(\begin{array}{c}
x\\
\tilde y
\end{array}\right) +\left(\begin{array}{c}
0\\
d
\end{array}\right) & \textrm{if} \quad x\geq 0, \vspace{0.2cm}\\
Z_0^-(x, \tilde y)=A^-\left(\begin{array}{c}
x\\
\tilde y
\end{array}\right) +\left(\begin{array}{c}
0\\
e
\end{array}\right) & \textrm{if} \quad x\leq 0,
\end{array}\right.
\end{equation*}
$A^+=\left(\begin{array}{cc}
a&b\\
c&-a\end{array}\right), A^-=\left(\begin{array}{cc}
0&-1\\
1&0\end{array}\right)$ with 
\[
\begin{array}{lll}
\rho  &=  &\sqrt{|\det(M_0^-)|}, \\ \\
a      &=  &\dfrac{1}{\rho}\Big(m_{11}^+ -\dfrac{m_{11}^- m_{12}^+}{m_{12}^-}\Big),\\ \\
b      &=   &-\dfrac{m_{12}^+}{m_{12}^-},\\ \\
c      &=    &\dfrac{1}{\rho^2}\left(\dfrac{(m_{11}^-)^{2}m_{12}^+}{m_{12}^-} -
2m_{11}^- m_{11}^+ - m_{12}^+ m_{21}^+\right),\\ \\
d     &=    &-\dfrac{m_{12}^+u_2^+}{\rho},\\ \\
e     &=   &-\dfrac{m_{12}^-u_2^-}{\rho}.
\end{array}
\]
Furthermore, the parameters of this canonical form satisfy
\begin{equation*}\label{cond-param-centro}
b<0, c>0, d>0, e>0, a^2+bc<0.
\end{equation*}
\end{proposition}

In general we can write
\begin{equation}\label{p1s}
Z_{\e}(x,y) = Z_0(x,y)+ \e Z_1(x,y)+\e^2 Z_2(x,y)+\CO(\e^3)
\end{equation}
where
\[
Z_1(x,y) =\left\{ \begin{array}{ll} B^+\left(\begin{array}{c}
x\\
y
\end{array} \right) + v^+  & \textrm{if} \quad x\geq 0,\vspace{0.2cm}\\
B^-\left(\begin{array}{c}
x\\
y
\end{array}\right) + v^- & \textrm{if} \quad x\leq 0,
\end{array}\right.
\]
and
\[
Z_2(x,y)=\left\{\begin{array}{lc} C^+\left(\begin{array}{c}
x\\
y
\end{array}\right) + w^+ & \textrm{if} \quad x\geq 0,\vspace{0.2cm}\\
C^-\left(\begin{array}{c}
x\\
y
\end{array}\right) + w^- & \textrm{if} \quad x\leq 0,
\end{array}\right.
\]
with $B^\pm=\left(\begin{array}{cc}
b_{11}^\pm&b_{12}^\pm\\
b_{21}^\pm&b_{22}^\pm\end{array}\right)$, \,\,
$v^\pm=\left(\begin{array}{cc}
v_1^\pm\\
v_2^\pm
\end{array}\right)$, \,\,
$C^\pm=\left(\begin{array}{cc}
c_{11}^\pm&c_{12}^\pm\\
c_{21}^\pm&c_{22}^\pm
\end{array}\right)$, \,\,
$w^\pm=\left(\begin{array}{cc}
w_1^\pm\\
w_2^\pm
\end{array}\right)$.  We denote the piecewise smooth vector field
$Z_\e$ by $Z_\e=(Z_\e^+,Z_\e^-)$, where $Z_\e^{\pm}$ is defined in
$\Sigma^{\pm}$.

\subsection{Crossing limit cycles}

For the piecewise linear differential system \eqref{p1s} we obtain an
upper bound for the number of crossing limit cycles bifurcating from
the periodic solutions of the unperturbed system.
We define
\[
\begin{array}{lll}
K_{0} & = & \dfrac{2 d \left(\xi ^2 (2 v_1^- b-b e
(b_{11}^-+b_{22}^-)
+2 v_1^+)+b d (b_{11}^++b_{22}^+)\right)}{e \xi },\\ \\
K_{1} & = & b e \xi ^2 (b_{11}^-+b_{22}^-),\\ \\
K_{2} & = & -\dfrac{b d^2 (b_{11}^++b_{22}^+)}{e \xi },
\end{array}
\]
where $ a^2 + b c = -\xi^{2}$, with $\xi > 0$. Notice that $(K_0,K_1,K_2)$ is surjective when it is seen as a function of the parameters $b_{11}^{\pm},$ $b_{22}^{\pm},$ and $v_1^{\pm}$. In fact, observe that $K_0$ is a linear function of $v_1^-, K_2$ is a linear function of $b_{11}^-$ and $K_2$ is a linear function of $v_{11}^+$.

\begin{mtheorem}\label{t1}
Suppose $Z_{\e}(x,y)$ is a piecewise linear vector field with two
zones separated by the straight line $x=0$ such that hypotheses
$(H_1), (H_2)$ and $(H_3)$ are satisfied by $Z_0$ and $\e>0$. Without loss of
generality, assume that $Z_{\e}$ writes as \eqref{p1s}. Then, the
following statements hold.
\begin{itemize}
\item[$(a)$] There exist elections of the parameters $K_0,$ $K_1,$
and $K_2$ for which system $Z_{\e}$ has any configuration of $1$,
$2$, or $3$ crossing limit cycles taking into account their
multiplicity.

\smallskip

\item[$(b)$] The highest amplitude limit cycle, when it exists and is hyperbolic,
and the infinity have opposite stability. More precisely, the highest
amplitude limit cycle is asymptotically stable $($resp. unstable$)$ and the infinity
is an unstable  $($resp. stable$)$ periodic orbit provided that
$\xi(b_{11}^-+b_{22}^-)+b_{11}^++b_{22}^+<0$ $($resp.
$\xi(b_{11}^-+b_{22}^-)+b_{11}^++b_{22}^+>0)$.

\smallskip

\item[$(c)$] The lowest amplitude limit cycle, when it exists and is hyperbolic,
is asymptotically stable $($resp. unstable$)$  provided that $b_{11}^-+b_{22}^-<0$,
or $b_{11}^-+b_{22}^-=0$ and $b v_1^-+v_1^+>0$ $($resp.
$b_{11}^-+b_{22}^->0$, or $b_{11}^-+b_{22}^-=0$ and $b
v_1^-+v_1^+<0)$.
\end{itemize}
\end{mtheorem}

Theorem \ref{t1} is proved in Subsection \ref{s3}.

We note that it is well known that piecewise linear differential
system with two regions
 separated by a straight line can exhibit $3$ crossing limit
cycles. The first numerical evidence of that was given in \cite{HY}
by Huan and Yang. Then, in \cite{LP}, Llibre and Ponce established an
analytical proof of the existence of  $3$ limit cycles. After that,
many authors have provided different examples of those systems having
$3$ limit cycles. For instance, in \cite{BPT} Buzzi et al., using a
seventh order piecewise linear perturbation, proved that 3 limit
cycles bifurcate from a linear center. In \cite{LNT}, Llibre et al.
proved that 3 limit cycles bifurcates from a piecewise linear center,
but using only a first order piecewise linear perturbation.  In the
present paper, Theorem \ref{t1} also guarantees the bifurcation of
$3$ crossing limit cycles after a first order perturbation. In the
next subsection we shall see that a second order perturbation allows
the study of sliding limit cycles.

\subsection{Sliding and escaping limit cycles}

We recall that a {\it sliding/escaping limit cycle} is a closed orbit
composed by segments of orbits of the sliding vector field $Z^s$ (see
Subsection \ref{subsecao-Filippov}) and $Z^+$ and/or $Z^-$. We must
mention that this kind of limit cycle has been considered before, see
for instance \cite{GP}. The main difference between sliding and escaping limit cycle is that, considering a initial condition on the sliding limit cycle, we have the uniqueness of trajectory for positive times and for escaping limit cycle we have uniqueness of trajectory for negative times. Moreover, the conclusions regarding escaping limit cycles can be obtained from the conclusions about sliding limit cycles just by reversing the arrow of time.

%

Consider a piecewise linear vector field $Z$ and assume that the
sliding/escaping region (see Subsection \ref{subsecao-Filippov}) is
bounded. A sliding/escaping limit cycle of $Z$ is a closed trajectory
composed by  trajectories of the sliding vector field $Z^s$, $Z^+,$
and/or $Z^-$. More precisely, we say that a sliding/escaping limit
cycle is: of {\it Type I} if it is composed by trajectories of $Z^s,$
and  $Z_{ \e}^+$ or $Z_{ \e}^-$; and of {\it  Type II} if it is
composed by trajectories of $Z^s$, $Z_{ \e}^+$ and $Z_{ \e}^-$, see
Figure \ref{GraficoSlidingCycleType}.

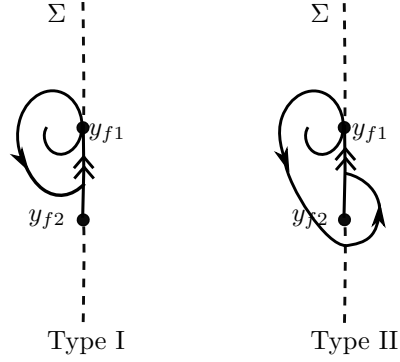
\begin{figure}[h]
\begin{center}
\begin{overpic}[width=5cm]{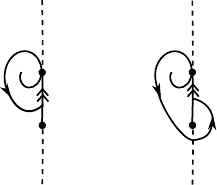}
\put(10,-7){Type I}\put(80,-7){Type
II}\put(10,80){$\Sigma$}\put(80,80){$\Sigma$}\put(21,50){$y_{f1}$}
\put(91,50){$y_{f1}$}\put(5,27){$y_{f2}$}\put(75,27){$y_{f2}$}
\end{overpic}
\end{center}
\caption{Types of sliding limit
cycles.}\label{GraficoSlidingCycleType}
\end{figure}

Since we are dealing with piecewise linear  vector fields, the
number of fold points for each vector field $Z^{+}_\e$ and $Z^{-}_\e$
is at most one. Therefore, the system can have at most one
sliding/escaping limit cycle.

In the following we denote $trace(C^-)$ by the trace of the matrix $C^-$ and by $\dfrac{(v_1^- b+v_1^+)^2}{b^2 e^2 \pi} = \Gamma$.

\begin{mtheorem}\label{t2}
Suppose $Z_{\e}(x,y)$ is a piecewise linear vector field with two
zones separated by the straight line $x=0$ such that hypotheses
$(H_1), (H_2)$ and $(H_3)$ are satisfied by $Z_0$. Without loss of
generality, assume that $Z_{\e}$ writes as \eqref{p1s}.
\begin{itemize}
\item [(1)] Additionally, suppose that $b_{11}^-=-b_{22}^-, b v_1^- +
v_1^+ < 0$ and $0<a<d+be$. Then, the following statements hold:
\begin{itemize}
\item[$(a)$]  If  $0<trace(C^-) < \dfrac{\Gamma}{2}$ then system $Z_{ \e}$ admits a sliding cycle of Type I.

\smallskip

\item[$(b)$] If  $\dfrac{\Gamma}{2} < trace(C^-) < 2\Gamma$ then system $Z_{
\e}$ admits a sliding cycle of Type II.
\end{itemize}

\smallskip

\item [(2)] Conversely, suppose that $b_{11}^-=-b_{22}^-, b v_1^-
+ v_1^+ > 0, a<0$ and $0<d+be$. Then, the following statements hold:
\begin{itemize}
\item[$(a)$]  If  $-\dfrac{\Gamma}{2} <
trace(C^-)< 0$ then system $Z_{ \e}$ admits a escaping cycle
of Type I.

\smallskip

\item[$(b)$] If  $trace(C^-) < - \dfrac{\Gamma}{2}$ then system $Z_{ \e}$ admits a escaping cycle of Type
II.
\end{itemize}

\item[(3)] Finally, suppose that $b_{11}^-\neq-b_{22}^-$. Then, $Z_{ \e}$ does not admit sliding limit cycle.
\end{itemize}
\end{mtheorem}

\begin{remark} We recall that the parameters of the unperturbed system $Z_0$ satisfy $b < 0, c>0, d>0, e>0$ and $a^2+ b c <0$. The inequalities on the parameters in the statement of Theorem \ref{t2} are related with the position of the fold points, the direction of the
sliding vector field, and the dynamics of the trajectories of $Z_\e^{\pm}$ passing through the fold points. More precisely, the condition $b_{11}^-=-b_{22}^-$ is a necessary condition for the existence of sliding/escaping limit cycle for the perturbed system $Z_{ \e}$, the conditions $b v_1^-+ v_1^+ < 0, b v_1^-+ v_1^+ > 0$ is a necessary conditions for the existence of sliding, escaping region, resp., the conditions $a<0$ and $0<d+be$ provides that the visible fold point is upper of the invisible fold point. Besides then, the conditions involving $trace(C^-)$ and $\Gamma$ is related with the intersection of the trajectories through the points $y_{f_1}, y_{f_2}$ and $y_{f_3}$ with the transverse section $\Lambda$ by the fold point $y_{f_1}$. For more details see Section \ref{secao-ciclosliding} and Figure \ref{S1S2S3}.
\end{remark}

Theorem \ref{t2} is proved in Subsection \ref{secao-ciclosliding}

\subsection{Simultaneity}

The next result provides conditions on the parameters of the vector
field $Z_{\e}$ for the simultaneous occurrence of crossing and
sliding/escaping limit cycles.

\begin{mtheorem}\label{t3}
Under the assumptions of Theorem \ref{t2} the following statements
hold.
\begin{itemize}
\item [$(1.a)$] If $0 < trace(C^-) < \dfrac{\Gamma}{2}$ then $Z_{ \e}$ possesses a sliding limit cycle of
type I and at most one more crossing limit cycle.

\smallskip

\item [$(1.b)$] If $\dfrac{\Gamma}{2} < trace(C^-) < 2 \Gamma$ then $Z_{ \e}$
possesses a sliding limit cycle of type II and at most one more
crossing limit cycle.

\smallskip

\item[$(2.a)$]  If  $-\dfrac{\Gamma}{2} <
trace(C^-) < 0$ then $Z_{ \e}$ possesses a escaping limit
cycle of type I and at most one more crossing limit cycle.

\smallskip

\item[$(2.b)$] If  $trace(C^-) < - \dfrac{\Gamma}{2}$ then $Z_{ \e}$ possesses a escaping limit cycle of
type II and at most one more crossing limit cycle.

\end{itemize}
Moreover, there exist parameters of $Z_{\e}$ for which the above
crossing limit cycle exists, see Figure
\ref{GraficoSlidingCrossingCycle}.
\end{mtheorem}

Theorem \ref{t3} is proved in Subsection \ref{secao-simultaniedade}.

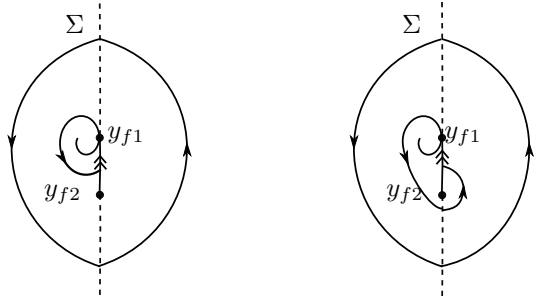
\begin{figure}[h]
\begin{center}
\begin{overpic}[width=7cm]{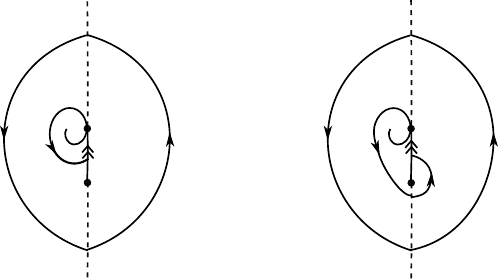}
\put(19,30){$y_{f1}$}\put(7,19){$y_{f2}$}\put(83,30){$y_{f1}$}
\put(72,19){$y_{f2}$}\put(11,50){$\Sigma$}\put(75,50){$\Sigma$}
\end{overpic}
\end{center}
\caption{Simultaneous occurrence of crossing and sliding limit cycles
for a 2-order linear perturbation of a piecewise linear
center.}\label{GraficoSlidingCrossingCycle}
\end{figure}


This paper is organized as follows. In Section
\ref{secao-teoria-basica} we provide some preliminary results.
Section \ref{s2} is devoted to the proofs of our main results.
Finally, in Section \ref{exemplos} we present some examples of
piecewise linear vector fields for which the upper bounds of crossing
and sliding/escaping limit cycles are reached.

\section{Preliminaries}\label{secao-teoria-basica}

This section is devoted to present some basic notions on nonsmooth
vector fields and all the tools needed to prove our main results.
Firstly, we introduce the concept of Filippov system. Then, some
results on extended Chebyshev systems are presented. This results are
important to bound the number of limit cycles. Finally, we discuss
the Bendixon transformation which allows the study of the stability
of the infinity.

\subsection{Filippov's convention}\label{subsecao-Filippov}

Let $h:\R^2\rightarrow\R$ a differentiable function for which $0$ is
a regular value. Consider the piecewise smooth vector field
\[
Z(x,y)=\left\{\begin{array}{ll}Z^+(x,y),&h(x,y)>0\\Z^-(x,y),&
h(x,y)<0.\end{array}\right.
\]
The switching manifold is given by $\Sigma=h^{-1}(0)$. Notice that
$h(x,y)=x$ for system \eqref{us1}. Consider the Lie's derivative
$Z^{\pm}h(p)=\langle Z^{\pm} , \nabla h \rangle,$ where $\langle
\cdot , \cdot \rangle$ is the canonical inner product in $\R^2$.
Then, according to Filippov's conventions (see \cite{F}), the
following regions on $\Sigma$ are distinguished: \textbf{Crossing
Region:} $\Sigma^c=\{p\in \Sigma;(Z^{+}h)(p).(Z^{-}h)(p)>0\}$.;
\textbf{Sliding Region:} $\Sigma^s=\{p\in \Sigma;(Z^{+}h)(p)<0,
(Z^{-}h)(p)>0\}$; and 
\textbf{Escaping Region:} $\Sigma^{e}=\{p\in
\Sigma;(Z^{+}h)(p)>0,(Z^{-}h)(p)<0\}$.
The local trajectories of the points in $\Sigma^s\cup\Sigma^e$ follow
the so called  \textit{sliding vector field}
\begin{equation}\label{eq campo filippov}
Z^s = \dfrac{Z^{-}h \cdot Z^{+} - Z^{+}h \cdot Z^{-}}{Z^{-}h -
Z^{+}h}.
\end{equation}
Then, the flow of $Z$ is obtained by the concatenation of flows of
$Z^{+},Z^{-}$, and $Z^s$.



We say that $p \in \Sigma$ is a {\it fold point} of $Z^{\pm}$ when
$Z^{\pm}h(p)=0$ and $(Z^{\pm})^{2} h(p) = Z^{\pm}.(Z^{\pm}h)(p)\neq
0$. Moreover, $p$ is called a {\it visible} (resp. {\it invisible})
fold point of $Z^{\pm}$ if $Z^{\pm}h(p)=0$ and
$(Z^{\pm})^{2}h(p)\gtrless 0$ (resp. $(Z^{\pm})^{2}h(p)\lessgtr 0$).


\begin{remark} Assuming that $p_0\in\Sigma^c$ is an invisible
fold point for both vector fields $Z^{+}$ and $Z^{-}$, a first return
map is well defined in a neighborhood of $p_0$. Indeed,  from the
Implicit Function Theorem, there exists a neighborhood $U$ of $p_0$
such that for each $p \in U\cap\Sigma,$ there exists a smallest
positive time $t(p)>0$ such that the trajectory $t\mapsto
\phi_{Z^{+}}(t,p)$ of $Z^{+}$ through $p$ intercepts $\Sigma$ at a
point $\widetilde{p}=\phi_{Z^{+}} (t(p),p)$. Then, define the
\textit{positive half-return map associated to $Z^{+}$} by
$\gamma_{Z^{+}}:(\R,0)\rightarrow (\R,0)$ where
$\gamma_{Z^{+}}(p)=\widetilde{p}$.
Analogously, we define the \textit{positive half-return map
associated to $Z^{-}$} by $\gamma_{Z^{-}}:(\R,0)\rightarrow (\R,0)$.
Thus, the first return map $\varphi_{Z}: (\Sigma,0)\rightarrow
(\Sigma,0)$ is defined by the composition
$\varphi_{Z}=\gamma_{Z^{-}}\circ \gamma_{Z^{+}}$.
\end{remark}

\subsection{Extended Chebyshev systems}

An ordered set of complex--valued functions
$\mathcal{F}=(g_0,g_1,\ldots,g_k)$ defined on a proper real interval
$I$ is an {\it Extended Chebyshev} system or ET--system on $I$ if and
only if any nontrivial linear combination of functions in
$\mathcal{F}$ has at most $k$ zeros counting multiplicities. The set
$\mathcal{F}$ is an {\it Extended Complete Chebyshev} system or an
ECT--system on $I$ if and only if for any $s\in\{0,1,\ldots,k\}$ we
have that $(g_0,g_1,\ldots,g_s)$ is an ET--system. For more details
see the book of Karlin and Studden \cite{KS}.

In order to prove that $\mathcal{F}$ is an ECT--system on $I$ it is
necessary and sufficient to show that $W(g_0,g_1,\ldots,g_s)(t)\neq
0$ on $I$ for $0\leq s\leq k$, where $W(g_0,g_1,\ldots,g_s)(t)$
denotes the Wronskian of the functions $(g_0,g_1,\ldots,g_s)$ with
respect to the variable $t$. That is,
\begin{equation*}
W(g_0,\ldots,g_s)(t)=\det\left(\begin{array}{ccc}
g_0(t)&\cdots&g_s(t)\\
g_0'(t)&\cdots&g_s'(t)\\
\vdots&\ddots&\vdots\\
g_0^{(s)}(t)&\cdots& g_s^{(s)}(t)
\end{array}\right).
\end{equation*}

In \cite{DT}, the authors proved the following results:

\begin{theorem}[\cite{DT}]\label{teorema-ECT}
Let $\mathcal{F} = [g_0, g_1, \dots , g_n]$ be an ordered set of
$C^\infty$ functions $g_j:[a, b]\to \R$ for $j=0,1,\ldots,n$ such
that there exists $\xi \in (a, b)$ with
$W(g_0,g_1,\ldots,g_{n-1})(\xi)=W_{n-1}(\xi) \neq 0$. Then,   the
following statements hold.
\begin{itemize}
\item [$(a)$] If $W_n(\xi) \neq  0$, then for each configuration of
$m \leq n$ zeros, taking into account their multiplicity, there
exists $f \in Span(\mathcal{F})$ with this configuration of zeros.

\item [$(b)$] If $W_n(\xi) = 0$ and $W_n'(\xi) \neq  0$, then for each
configuration of $m \leq n+1$ zeros, taking into account their
multiplicity, there exists $f \in  Span(\mathcal{F})$ with this
configuration of zeros.
\end{itemize}
\end{theorem}

\begin{corollary}[\cite{DT}]
Let $\mathcal{F} = [g_0, g_1, \dots , g_n]$ be an ordered set of
$C^\infty$ functions $g_j:[a, b]\to \R$ for $j=0,1,\ldots,n$. Assume
that all the Wronskians are nonvanishing except $W_{n}(x)$, which has
exactly one zero on $(a, b)$ and this zero is simple. Then,
$Z(\mathcal{F}) = n +1$ and for any configuration of $m \leq n +1$
zeros there exists an element in $Span(\mathcal{F})$ realizing it.
\end{corollary}

\subsection{The Bendixson transformation}\label{sec:bendisxon}
The Bendixson transformation is a useful tool to analyze the
stability of the infinity of planar vector fields. In what follows,
following \cite{GLN,Llibre-Ponce}, we shall discuss this
transformation. Consider the differential systems
\begin{equation} \label{sistema-planar}
\dot{x}=f(x,y,\e),\qquad \dot{y}=g(x,y,\e),
\end{equation}
where $f,g$ are Lipschitz functions in the variables $(x, y)$ and $\e
> 0$ is a small parameter. Applying to system \eqref{sistema-planar} the
{\it Bendixson transformation} defined as
\begin{equation}
\begin{array}{ll}
\left(
\begin{array}{c}
u\\
v
\end{array} \right) &= \dfrac{1}{x^2+y^2}\left(
\begin{array}{c}
x\\
y
\end{array} \right),
\end{array}\label{transf-Bendixson}
\end{equation}
we obtain an equivalent system whose local phase portrait at the
origin is equivalent to the local phase portrait of system
\eqref{sistema-planar} in a neighborhood of the infinity.

Composing the Bendixson change of variables \eqref{transf-Bendixson}
with the polar coordinates $u = r \cos\theta$, $v = r \sin\theta$, we
get the {\it polar Bendixson transformation} $x= (\cos\theta)/r$, $y=
(\sin\theta)/r$. Applying this last transformation, system
\eqref{sistema-planar} becomes
\begin{equation}\label{bendixson2}
\begin{array}{lll}
\dot{r}=& R(r,\theta,\e)=&-r^{2}\left[ f\left(\dfrac{\cos\theta}{r},
\dfrac{\sin\theta}{r},\e\right) \cos\theta +
g\left(\dfrac{\cos\theta}{r}, \dfrac{\sin\theta}{r},\e\right)
\sin\theta\right],\vspace{0.2cm}\\
\dot{\theta}=& \xi(r,\theta,\e)=&-r\left[
f\left(\dfrac{\cos\theta}{r},\dfrac{\sin\theta}{r},\e\right)
\cos\theta - g\left(\dfrac{\cos\theta}{r},\dfrac{\sin\theta}{r},
\e\right) \cos\theta\right].
\end{array}
\end{equation}
We shall study the flow of system \eqref{bendixson2} contained in the
half-cylinder $\mathbb{R}^{+}\times\mathbb{S}^{1} = \{ (r,\theta):r
\geq 0,$ $\theta\in (-\pi,\pi) \}$.  Notice that after multiplying
\eqref{bendixson2} by a power of $r$, the system can be extended for
$r=0$. Therefore, the existence of a periodic orbit at infinity for
system \eqref{sistema-planar} is equivalent to the existence of the
periodic orbit $r = 0$ for system \eqref{bendixson2} on the cylinder.
Now, consider the assumptions:
\begin{itemize}
\item[($a$)] the functions $R$ and $\xi$ are locally Lipschitz functions
in the variable $r$ and they have period $2\pi$ in the variable
$\theta$.

\item[($b$)] $R(0,\theta,\e)=0$ and $\xi(0,\theta,\e)\neq 0$ for
all $\theta\in\mathbb{S}^{1}$ and for every $\e\geq 0$ sufficiently
small.
\end{itemize}
Notice that $(a)$ and $(b)$ are sufficient and necessary conditions
in order to guarantee that system \eqref{sistema-planar} has a
periodic solution at infinity.  Finally, taking $\theta$ as the new
independent variable the differential system \eqref{bendixson2} can
be written as the first order differential equation
\begin{equation}\label{eqbendixson3}
r'=\dfrac{dr}{d\theta}=S(r,\theta,\e)=\dfrac{R(r,\theta,\e)}{\xi
(r,\theta,\e)}.
\end{equation}
Consequently, the Poincar\'{e} map defined on a neighborhood of $r=0$
is given by $\Pi(\rho)=r(2\pi,\rho,\e)$, where $r(\theta,\rho,\e)$ is
the solution of \eqref{eqbendixson3} such that $r(0,\rho,\e)=\rho$.

\section{Proof of the main results}\label{s2}

In this section we provide the proofs of Proposition \ref{p1} and
Theorems \ref{t1}, \ref{t2} and \ref{t3}.

\subsection{{Proof of Proposition \ref{p1}}}

We assume that the piecewise linear vector field $W_0(x,y)$ satisfies
the hypotheses $(H_1), (H_2)$ and $(H_3)$. Then,   the left and right
linear differential systems are written as $(\dot x, \dot
y)=M_0^{\pm}(x,y)+(0,u_2^{\pm})^T,$ where
\[
M_0^{\pm}=\left(\begin{array}{cc}
m_{11}^{\pm}&m_{12}^{\pm}\\
m_{21}^{\pm}&-m_{11}^{\pm}\end{array}\right),
\]
with $(m_{11}^{\pm})^2+m_{12}^{\pm} m_{21}^{\pm}<0$ and
$m_{12}^{\pm}\neq 0$.

Then, applying the change of variables $(\tilde x,\tilde y)=\psi(x,y)
=(x, -m_{11}^-x-m_{12}^-y)$ we obtain the following piecewise linear
differential system
\begin{equation}\label{key2}
\left(\begin{array}{c}
\tilde x'\\
\tilde y'
\end{array}\right)
=\left\{\begin{array}{LC} \widetilde{A}^+\left(\begin{array}{c}
\tilde x\\
\tilde y
\end{array}\right)
+\left(\begin{array}{c}
0 \\
-m_{12}^+u_2^+
\end{array}\right) & \textrm{if} \quad \tilde x\geq 0, \vspace{0.2cm}\\
\widetilde{A}^-\left(\begin{array}{c}
\tilde x\\
\tilde y
\end{array}\right)
+\left(\begin{array}{c}
0 \\
-m_{12}^+u_2^-
\end{array}\right)  & \textrm{if} \quad \tilde x\leq 0,
\end{array}\right.
\end{equation}
where
\[
\widetilde{A}^+=\left(\begin{array}{cc}
m_{11}^+ - \dfrac{m_{11}^- m_{12}^+}{m_{12}^-}& -\dfrac{m_{12}^+}{m_{12}^-}\\
\dfrac{(m_{1}^-)^{2}m_{12}^+}{m_{12}^-} - 2m_{11}^- m_{11}^+ -
m_{12}^- m_{21}^+ & \dfrac{m_{11}^-m_{12}^+}{m_{12}^-} - m_{11}^+
\end{array}\right)
\]
and
\[\widetilde{A}^-=\left(\begin{array}{cc}
0&-1\\
-(m_{11}^-)^{2} - m_{12}^-m_{21}^-&0\end{array}\right).
\]
Notice that the above change of variables fixes the switching
manifold.

Now, let $\rho=\sqrt{|(m_{11}^-)^{2} + m_{12}^- m_{21}^-|}$, where
$(m_{11}^-)^{2} + m_{12}^- m_{21}^-<0$. Proceeding with the following
change of variables and rescaling of time
\[ (\tilde{x},\tilde{y},\tilde{t})  \mapsto \left(
\dfrac{x}{\rho},y ,\dfrac{t}{\rho}\right)\]
 system \eqref{key2} becomes
\begin{equation*}
Z_0^{-}(x,\tilde y)= \left(\begin{array}{cc}
0 & -1\\
1 & 0
\end{array}\right)\left(\begin{array}{c}
x\\
\tilde y
\end{array}\right)
+\left(\begin{array}{c}
0 \\
e
\end{array}\right),\quad\text{for}\quad x\leq 0
\end{equation*}
and
\begin{equation*}
Z_0^{+}(x,\tilde y)= \left(\begin{array}{cc}
a & b\\
c & -a
\end{array}\right)\left(\begin{array}{c}
x\\
\tilde y
\end{array}\right)
+\left(\begin{array}{c}
0 \\
d
\end{array}\right),\quad \text{for}\quad  x\geq0,
\end{equation*}
where
\[
\begin{array}{l}
a=\dfrac{1}{\rho}\Big(m_{11}^+ - \dfrac{m_{11}^-
m_{12}^+}{m_{12}^-}\Big),\vspace{0.2cm} \\
b=-\dfrac{m_{12}^+}{m_{12}^-},\vspace{0.2cm}\\
c=\dfrac{1}{\rho^2}\Big(\dfrac{(m_{11}^-)^{2}m_{12}^+}{m_{12}^-} -
2m_{11}^- m_{11}^+ - m_{12}^+ m_{21}^+\Big),\vspace{0.2cm}\\
d=-\dfrac{m_{12}^+u_2^+}{\rho} \quad \text{and}\vspace{0.2cm}\\
e=-\dfrac{m_{12}^-u_2^-}{\rho} > 0.
\end{array}
\]
The singular points of $Z_0^-, Z_0^+$ are given by $p^-=(-e,0)$ and
$p^+=\dfrac{d}{a^2+bc}(-b,a)$. From $(H_1), (H_2)$ and $(H_3)$
we conclude that $b<0, c>0, d>0, e>0$ and $a^2+bc<0$. $\hfill\square$

\subsection{Study of the infinity}\label{secao-estudo-infinito}

Applying the Bendixson change of coordinates given in  \eqref{p1s} to
$Z_{\varepsilon}$, we obtain that the differential system in
$\Sigma^-$ can be written
\begin{equation*}
\begin{array}{ll}
\dfrac{du^-}{dt}=& -v (2 e u+1) \Big(u^2+v^2\Big) + \varepsilon
\Big(-u^3 (v_1^- u+b_{11}^-)+v_1^- v^4+u v^2 (b_{11}^--2 b_{22}^-)\\
&\left. -u^2 v (a_{2}+2 (v_2^- u+b_{21}^-))+ v^3 (a_{2}-2 v_2^-
u)\Big)\right|_{u=u^-, v=v^-},\vspace{0.2cm} \\
\dfrac{dv^-}{dt}=& e u^4-e v^4+u^3+u v^2
+ \varepsilon  \Big(u^2 v (-2 v_1^- u-2 b_{11}^-+b_{22}^-)-v^3 (2 v_1^- u+b_{22}^-) \\
&\left. -u v^2 (2 a_{2}+b_{21}^-)+u^3 (v_2^- u+b_{21}^-)-v_2^-
v^4\Big)\right|_{u=u^-, v=v^-},
\end{array}
\end{equation*}
and for the differential system in $\Sigma^+$ can be written as
\begin{equation*}
\begin{array}{ll}
\dfrac{du^+}{dt}=& -a u^3+3 a u v^2-u^2 v (b+2 (c+d u))+v^3 (b-2 d u)
+\varepsilon  \Big(-u^3 (v_1^+ u+b_{11}^+)\\
&\left. +v_1^+ v^4 +u v^2 (b_{11}^+-2 b_{22}^+)-u^2 v (c_{2} +2
(d_{0} u+d_{1}))+v^3 (c_{2}-2 d_{0} u)\Big)\right|_{u=u^+, v=v^+},\vspace{0.2cm} \\
\dfrac{dv^+}{dt}=& -3 a u^2 v+a v^3-u v^2 (2 b+c)+u^3 (c+d u) +
\varepsilon  \Big(u^2 v (-2 v_1^+ u -2 b_{11}^++b_{22}^+)\\
&\left. -v^3 (2 v_1^+ u+b_{22}^+) -u v^2 (2 c_{2}+d_{1})+u^3 (d_{0}
u+d_{1})-d_{0} v^4\Big)-d v^4\right|_{u=u^+, v=v^+}.
\end{array}
\end{equation*}
Applying the polar change of coordinates the left differential system
can be written as
\begin{equation}\label{r-infinity}
\begin{array}{ll}
\dfrac{dr^-}{dt}=& -e r^2 \sin \theta -r \varepsilon  \Big(\cos
\theta (v_1^- r+(a_{2}+b_{21}^-) \sin \theta)+b_{11}^- \cos ^2
\theta\\
& \left. +\sin \theta (v_2^- r+b_{22}^- \sin \theta)\Big)\right|_{
r=r^-, \theta=\theta^-},\vspace{0.2cm} \\
\dfrac{d\theta^-}{dt}=& 1+e r \cos \theta+ \varepsilon  \Big(-\sin
\theta (v_1^- r+a_{2} \sin \theta)+\cos \theta ((b_{22}^--b_{11}^-)
\sin \theta\\
& \left. +v_2^- r)+b_{21}^- \cos ^2\theta\Big)\right|_{ r=r^-,
\theta=\theta^-},
\end{array}
\end{equation}
and, similarly, the right differential system can be written as
\begin{equation}\label{r+infinity}
\begin{array}{ll}
\dfrac{dr^+}{dt}=& -r (a \cos (2 \theta )+\sin \theta ((b+c)
\cos \theta+d r))-r \varepsilon  \Big(\cos \theta (v_1^+ r\\
&\left. +(c_{2}+d_{1}) \sin \theta) +b_{11}^+ \cos ^2\theta+\sin
\theta (d_{0} r+b_{22}^+ \sin \theta)\Big)\right|_{
r=r^+, \theta=\theta^+},\vspace{0.2cm} \\
\dfrac{d\theta^+}{dt}=& \cos \theta (d r-2 a \sin \theta)-b \sin
^2\theta+c \cos ^2\theta+\varepsilon  \Big(-\sin \theta
(v_1^+ r+c_{2} \sin \theta)\\
&\left. +\cos \theta ((b_{22}^+-b_{11}^+) \sin \theta+d_{0} r)+d_{1}
\cos ^2\theta\Big)\right|_{ r=r^+, \theta=\theta^+}.
\end{array}
\end{equation}
Notice that $r=0$ is a periodic solution for both differential systems \eqref{r-infinity} and \eqref{r+infinity}. Considering the rescaling $r^{\pm}=\varepsilon^{3}\rho^{\pm}$ and taking $\theta^{\pm}$ as the new independent variable, we obtain
\begin{equation}\label{rho-infinity}
\begin{array}{ll}
\dfrac{d\rho^-}{d\theta^-}= & -\e \rho  \Big(b_{11}^- \cos ^2\theta
+(a_{2}+b_{21}^-) \sin \theta \cos \theta+b_{22}^- \sin ^2(\theta
)\Big) \\
& +\e^2\rho \Big((b_{22}^- -b_{11}^-) \sin \theta \cos \theta-a_{2}
\sin ^2\theta+ b_{21}^- \cos ^2\theta\Big) \\
&\left. \Big(b_{11}^- \cos ^2\theta+(a_{2}+b_{21}^-) \sin \theta \cos
\theta +b_{22}^- \sin ^2\theta\Big)+ \mathcal{O}(\varepsilon^3)
\right|_{\rho=\rho^-, \theta=\theta^-},
\end{array}
\end{equation}
and
\begin{equation}\label{rho+infinity}
\begin{array}{ll}
\dfrac{d\rho^+}{d\theta^+}= & \dfrac{\rho  (2 a \cos (2 \theta )+
(b+c) \sin (2 \theta ))}{2 \Big(a \sin (2 \theta )+b \sin ^2\theta -c
\cos ^2\theta\Big)}-\e\dfrac{\rho}{2 \Big(a \sin (2 \theta )+
b \sin ^2\theta-c \cos ^2\theta\Big)^2}\\
&  \Big(-\sin (2 \theta ) (a (b_{11}^++b_{22}^+) +b d_{1}-
c c_{2})+\cos (2 \theta ) (a c_{2}-a d_{1}+b b_{22}^++c b_{11}^+)\\
&\left. -a c_{2}-a d_{1} -b b_{22}^++c b_{11}^+\Big) +
\mathcal{O}(\varepsilon^2)\right|_{\rho=\rho^+, \theta=\theta^+}.
\end{array}
\end{equation}
Let $\rho^{\pm}_{\e}(\theta)=\rho_{0}^{\pm}(\theta) +
\varepsilon\rho_{1}^{\pm}(\theta,\e) + \mathcal{O}(\varepsilon^2)$ be, respectively,
the solutions of \eqref{rho-infinity} and \eqref{rho+infinity} satisfying  $\rho^\pm\left( -\pi/2,\e\right) =\rho_{0}$.
Thus, we have
\begin{equation*}
\begin{array}{ll}
\rho_{0}^{-}(\theta)= & \rho_{0},\\ \\
\rho_{0}^{+}(\theta)= & \dfrac{\rho_{0}}{\sqrt{2}} \Big( \sqrt{
\dfrac{2 a \sin (2 \theta )-(b+c) \cos (2 \theta )+b-c}{b}} \Big),\vspace{0.2cm} \\
\rho_{1}^{-}(\theta)= & \dfrac{\rho_{0}}{4}  (-2 b_{11}^- \theta
-b_{11}^- \sin (2 \theta )+\pi  b_{11}^-+a_{2} \cos (2 \theta
)+a_{2}+b_{21}^- \cos
(2 \theta )+b_{21}^--2 b_{22}^- \theta\\
&  +b_{22}^- \sin (2 \theta ) +\pi  b_{22}^-),\\ \\
\rho_{1}^{+}(\theta)= & - \dfrac{\rho_{0}}{4b\xi \sqrt{-2b}
\sqrt{-2 a \sin (2 \theta )+(b+c) \cos (2 \theta )-b+c}} \Big(2 b
(b_{11}^++b_{22}^+)
\arctan \Big(\dfrac{a}{\xi}\\
& +\dfrac{b \tan \theta}{\xi }\Big) (-2 a \sin (2 \theta )+(b+c)
\cos (2 \theta )-b+c)+2 \sin (2 \theta ) (\pi  a b (b_{11}^++b_{22}^+)\\
& +2 a c_{2} \xi +b \xi  (b_{22}^+-b_{11}^+)) -\cos (2 \theta )
(\pi  b (b+c) (b_{11}^++b_{22}^+)+2 \xi  (c c_{2}-b d_{1}))\\
&+\pi  b (b-c) (b_{11}^++b_{22}^+)+2 b d_{1} \xi -2 c c_{2} \xi
\Big),
\end{array}
\end{equation*}
where $ a^2 + b c = -\xi^{2}$, with $\xi > 0$. Therefore, the
displacement function writes
\begin{equation*}
\begin{array}{ll}
\rho_{\e}(\rho_0)& = \rho_{\e}^{+}(\pi/2)-\rho_{\e}^{-}(-3\pi/2) \vspace{0.2cm} \\
& =  -\dfrac{1}{2} \varepsilon\pi\rho_{0} \Big(b_{11}^-+b_{22}^- +
\dfrac{b_{11}^+  + b_{22}^+}{\xi} \Big)+\mathcal{O}(\varepsilon^{2}).
\end{array}
\end{equation*}
Now, since $r=0$ is a periodic solution for both differential equations \eqref{r-infinity} and \eqref{r+infinity}, we conclude that $\rho_{\e}(0)=0$. Following Section \ref{sec:bendisxon}, this means that the infinity can be seen as a periodic solution of $Z_{\e}$. To investigate its stability, we compute
\[
\rho_{\e}'(0)=-\dfrac{1}{2} \varepsilon\pi \Big(b_{11}^-+b_{22}^- +
\dfrac{b_{11}^+  + b_{22}^+}{\xi} \Big)+\mathcal{O}(\varepsilon^{2}).
\]
Therefore, since $\xi>0$, we get that $\sgn(\rho_{\e}'(0))=-\sgn\left(\xi(b_{11}^-+b_{22}^- )+b_{11}^+  + b_{22}^+\right),$ for $\e>0$ sufficiently small.
Consequently, by construction, if $\xi(b_{11}^-+b_{22}^- )+b_{11}^+  + b_{22}^+ > 0$
(resp. $\xi(b_{11}^-+b_{22}^-) + b_{11}^+  + b_{22}^+ < 0$), then the
infinity is a asymptotically stable (resp. unstable) periodic solution.

\subsection{Proof of Theorem \ref{t1}}\label{s3}

The proof will be split in three steps. In the first one we prove
that the number of crossing limit cycles of $Z_{1,\e}(X)$ is given by
the zeros of the first order Melnikov function
\begin{equation}\label{M1}
\begin{array}{ll}
M_{1}(y_0)= & \dfrac{1}{2 y_{0}}\Big(4 v_1^- y_{0}-2
(b_{11}^-+b_{22}^-)
\Big(\pi  \Big(e^2+y_{0}^2\Big)+e y_{0}\Big)\vspace{0.2cm} \\
& +(b_{11}^-+b_{22}^-) \Big(e^2+y_{0}^2\Big) \arccos\Big(\frac{2 e^2}
{e^2+y_{0}^2}-1\Big)\vspace{0.2cm} \\
& -\dfrac{1}{b \xi ^3} \Big(-2 b d y_{0} \xi  (b_{11}^++b_{22}^+)-
4 v_1^+ y_{0} \xi ^3+b (b_{11}^++b_{22}^+) \Big(d^2+y_{0}^2 \xi ^2\Big)\vspace{0.2cm} \\
& \arccos\Big(\frac{2 d^2}{d^2+y_{0}^2 \xi ^2}-1\Big) \Big)\Big),
\end{array}
\end{equation}
where the Melnikov Function is given by
\[
M(y_{0},\varepsilon)=M_0(y_0) + M_1(y_0)\varepsilon +
M_2(y_0)\varepsilon^2 +\mathcal{O}(\varepsilon)^3
\]
and $M_i(y_0)=M^{-}_{i}(y_0)-M^{+}_{i}(y_0)$ for $i=0,1,2$.

\vspace{0.2cm}

In the second one we prove that the upper bound of the number of
zeros is three, and that this number is reached. Finally, in the
third step we study the stability of the crossing limit cycles.

\begin{lemma}\label{mel}
The zeros of $M_1(y_0)$ correspond to crossing limit cycles for
$Z_{1, \e}(X)$.
\end{lemma}

\begin{proof}

Let $Z_{\e}$ be given by \eqref{p1s}. From hypotheses we know that $b
< 0,d > 0$ and $a^2 + b c <0$. Thus, denote $a^2 + b c=-\xi^2$,
$\xi>0$. Let $(x^{\pm}_{\e}(t),y^{\pm}_{\e}(t))$ be the trajectories
of the linear vector fields $Z_0^{\pm}$ satisfying $x^-_{\e}(0) = 0,$
$y^-_{\e}(0) = y_0>0$, and $x^+_{\e}(0) = 0, y^+_{\e}(0) = y_1$. So,
for $\e=0$, we compute
\begin{equation}\label{leftsolution1}
\begin{array}{l}
x^-_{0} (t) = e (-1+\cos t)-y_0 \sin t,\vspace{0.2cm}\\
 y^-_{0} (t) = y_0 \cos t+e \sin t,
\end{array}
\end{equation}
and
\begin{equation}\label{rightsolution1}
\begin{array}{l}
x^+_{0}(s) = \dfrac{b (d-d \cos(s \xi )+y_{1} \xi  \sin(s \xi ))}{\xi
^2},
\vspace{0.2cm}\\
y^+_{0}(s) = \dfrac{1}{\xi ^2}\Big(-d+\Big(d+y_{1} \xi ^2\Big) \cos(s
\xi )+\xi (d-y_{1})\sin(s \xi)\Big).
\end{array}
\end{equation}
Let $t_{l,\e}>0$ and $t_{r,\e}<0$ be the first return times to
$\Sigma$ of the above solutions, that is
$x^-_{\e}(t_l)=x^+_{\e}(t_r)=0$. For $\e=0$ we have
\begin{equation*}
t_{l0}=2 \pi -\arccos\Big(\frac{2 e^2}{e^2+y_{0}^2}-1 \Big)
\end{equation*}
and
\[
t_{r0}=-\dfrac{1}{\xi}\arccos\left(\frac{2 d^2}{d^2+\xi ^2
y_{1}^2}-1\right).
\]
Writing $t_{l,\e}=t_{l0} +t_{l1} \e+\CO(\e^2)$ and $t_{r,\e}=t_{r0}
+t_{r1} \e+\CO(\e^2)$, the coefficients $t_{l1}$ and $t_{r1}$  can be
computed by expanding the equations $x^-_{\e}(t_l)=0$ and
$x^+_{\e}(t_r)=0$ around $\e=0$. So
\[
\begin{array}{ll}
t_{l1}=& \dfrac{1}{2 y_{0} \Big(e^2+y_{0}^2\Big)}\Big(2 y_{0} (2
v_1^- e-e (e (b_{11}^-+b_{22}^-)+y_{0} (b_{21}^--a_{2}))+
2 v_2^- y_{0})-2 \pi  \Big(e^2+y_{0}^2\Big) (e (b_{11}^-\vspace{0.2cm} \\
& +b_{22}^-)+y_{0} (b_{21}^--a_{2}))+\Big(e^2+y_{0}^2\Big)
\arccos\Big(\frac{2 e^2}{e^2+y_{0}^2}-1\Big) (e
(b_{11}^-+b_{22}^-)+y_{0} (b_{21}^--a_{2}))\Big),
\end{array}
\]
and
\[
\begin{array}{ll}
t_{r1}=& \dfrac{1}{2 b y_{1} \xi ^3 \Big(d^2+y_{1}^2 \xi ^2\Big)}
\Big(-2 d y_{1} \xi  \Big(y_{1} \Big(a^2 c_{2}+a b (b_{22}^+
-b_{11}^+)-b^2 d_{1}\Big)+b d (b_{11}^++b_{22}^+)\Big)\vspace{0.2cm} \\
& +2 y_{1} \xi ^3 (2 a v_1^+ y_{1}+2 b d_{0} y_{1}-2 v_1^+ d-c_{2} d
y_{1})+\Big(c_{2} y_{1} \Big(a^2+\xi ^2\Big)+b (-a b_{11}^+ y_{1}
+a b_{22}^+ y_{1}\vspace{0.2cm} \\
& +b_{11}^+ d +d b_{22}^+)-b^2 d_{1} y_{1}\Big)\Big(d^2+y_{1}^2 \xi
^2\Big) \arccos\Big(\dfrac{2 d^2}{d^2+y_{1}^2 \xi ^2}-1\Big) \Big).
\end{array}
\]
Replacing the expression of $t_{l0}$ and $t_{l1}$ in the expansion of
the solution of \eqref{leftsolution1} we get the positive half return
map in $\Sigma^-$, i.e. 
\begin{equation*}\label{ml1}
M^{-}_{1}=\frac{4 v_1^- y_{0}-2 (b_{11}^-+b_{22}^-) \Big(\pi  e^2+e
y_{0}+\pi  y_{0}^2\Big)+(b_{11}^-+b_{22}^-) \Big(e^2+y_{0}^2\Big)
\arccos \Big(\frac{2 e^2}{e^2+y_{0}^2}-1\Big)}{2 y_{0}}.
\end{equation*}
Analogously, replacing the expressions of $t_{r0}$ and $t_{r1}$ in
the expansion of the solution of \eqref{rightsolution1} we obtain the
negative half return map in $\Sigma^+$, namely
\begin{equation*}\label{mr1}
M^{+}_{1}=\frac{b (b_{11}^++b_{22}^+) \left(d^2+\xi ^2 y_{1}^2\right)
\arccos\left(\frac{2 d^2}{d^2+\xi ^2 y_{1}^2}-1\right)-2 b d \xi
y_{1} (b_{11}^++b_{22}^+)-4 v_1^+ \xi ^3 y_{1}}{2 b \xi ^3 y_{1}}.
\end{equation*}
The difference between $M^{-}_{1}$ and $M^{+}_{1}$ provides the first
order Melnikov function $M_1(y_0) = M^{-}_{1}(y_0) - M^{+}_{1}(y_0)$
given in \eqref{M1}, and the simple zeros of $M_1(y_0)$ provide the
crossing limit cycles of $Z_{1,\e}(X)$.
\end{proof}

\begin{lemma}\label{[lema2]}
The function $M_{1}$ presented in \eqref{M1} has at most three simple
zeros. Furthermore,  this upper bound is reached.
\end{lemma}

\begin{proof}
Considering the change of coordinates and parameters given by
$y_{0}=\alpha s_{0}/\xi$, $e=\alpha/\xi$ and $d=\alpha\beta$ in the
function $M_1$, given in \eqref{M1}, we obtain
\begin{equation*}
\begin{array}{ll}
M_1(s_0)=   & -\dfrac{1}{2 b \beta  s_{0} \xi ^2}\Big(-4 \beta s_{0}
\xi ^2 (v_1^- b+v_1^+)-2 \alpha  b \beta ^2 s_{0}(b_{11}^+
+b_{22}^+)+\alpha  b \xi  (b_{11}^-+b_{22}^-) \Big(\pi  \beta ^2
s_{0}^2\vspace{0.2cm}\\
&+2 \beta  s_{0}+\pi \Big)+\alpha  b \xi  (b_{11}^- +b_{22}^-)
\Big(\beta ^2 s_{0}^2+1\Big) \arccos\Big(1-\dfrac{2}{\beta ^2
s_{0}^2+1}\Big)+\alpha  b \beta ^2 \Big(s_{0}^2+1\Big)\vspace{0.2cm}\\
& (b_{11}^++b_{22}^+) \arccos\Big(\dfrac{2}{s_{0}^2+1}-1\Big)\Big).
\end{array}
\end{equation*}
The positive zeros of $M_1(s_0)$ coincide with the zeros of $2 b
\beta  s_{0} \xi ^2 M_1(s_0)=\widetilde{M}_{1}(s_0)$. We have that
\begin{equation*}
\begin{array}{ll}
\widetilde{M}_{1}(s_0)=& 2 \alpha  b \beta ^2 s_{0}
(b_{11}^++b_{22}^+) +4 \beta  s_{0} \xi ^2 (v_1^- b+v_1^+)-\alpha  b
\xi  (b_{11}^-+b_{22}^-)
\Big(\pi  \beta ^2 s_{0}^2+2 \beta  s_{0}+\pi \Big)\vspace{0.2cm}\\
& -\alpha  b \xi  (b_{11}^-+b_{22}^-) \Big(\beta ^2 s_{0}^2+1\Big)
\arccos\Big(1-\dfrac{2}{\beta ^2 s_{0}^2+1}\Big)-\alpha  b \beta ^2
\Big(s_{0}^2+1\Big) (b_{11}^++b_{22}^+)\vspace{0.2cm}\\
& \arccos\Big(\dfrac{2}{s_{0}^2+1}-1\Big).
\end{array}
\end{equation*}
If
\begin{equation*}
\begin{array}{ll}
K_{0}=& 2 \beta  \Big(2 \xi ^2 (v_1^- b+v_1^+)-\alpha  b \xi
(b_{11}^-+b_{22}^-)+\alpha  b \beta  (b_{11}^++b_{22}^+)\Big),\vspace{0.2cm} \\
K_{1}=& \alpha  b \xi  (b_{11}^-+b_{22}^-),\vspace{0.2cm} \\
K_{2}=& - b \alpha\beta ^2 (b_{11}^++b_{22}^+),
\end{array}
\end{equation*}
then $\widetilde{M}_{1}$ can be rewritten as
\begin{equation*}
\widetilde{M}_{1}(s_0)= K_{0} s_{0}+K_{1} \Big(\beta ^2
s_{0}^2+1\Big) \Big(\arccos\Big(\dfrac{2}{\beta ^2
s_{0}^2+1}-1\Big)-2 \pi \Big)+K_{2} \Big(s_{0}^2+1\Big)
\arccos\Big(\dfrac{2}{s_{0}^2+1}-1\Big).
\end{equation*}
Note that $\widetilde{M}_{1}(s_0)$ is the linear combination
\begin{equation*}
\widetilde{M}_{1}(s_0)=K_{0} f_{0}(s_0) -2\pi K _{1} f_{1}(s_0) +
K_{2} f_{2}(s_0)+  K_{1} f_{3}(s_0)
\end{equation*}
of the functions
\begin{equation*}\label{fi}
\begin{array}{lll}
f_{0}(s_0) & = & g_{0}\vspace{0.2cm}\\
f_{1}(s_0) & = & 1 + \beta ^2 s_{0}^2\vspace{0.2cm}\\
f_{2}(s_0) & = & \left(s_{0}^2+1\right) \arccos\left(\dfrac{2}
{s_{0}^2+1}-1\right)\vspace{0.2cm}\\
f_{3}(s_0) & = & (\beta ^2 s_{0}^2+1)\arccos\Big(\dfrac{2}{\beta ^2
s_{0}^2+1}-1\Big).
\end{array}
\end{equation*}
Denoting $W_{k}(s_{0})=W_{k}(f_{0}, f_{1}, . . . , f_{k})(s_{0})$, we
have
\begin{equation*}\label{Wi}
\begin{array}{lll}
W_{0}(s_0) & = & s_{0}\vspace{0.2cm}\\
W_{1}(s_0) & = & \beta ^2 s_{0}^2-1\vspace{0.2cm}\\
W_{2}(s_0) & = & 2 \left(\beta ^2-1\right) \arccos\left(\dfrac{2}
{s_{0}^2+1}-1\right)-\dfrac{4 \left(\beta ^2+1\right) s_{0}}
{s_{0}^2+1}\vspace{0.2cm}\\
W_{3}(s_0) & = & \dfrac{16 \beta ^3 \left(\beta ^2-1\right) \left(2
s_{0} \left(s_{0}^2-1\right)+\left(s_{0}^2+1\right)^2
\arccos\left(\dfrac{2}{s_{0}^2+1}-1\right)\right)}{\left(s_{0}^2+
1\right)^2 \left(\beta ^2 s_{0}^2+1\right)^2}.
\end{array}
\end{equation*}
Observe that the functions $W_{k}(s_{0})$ for $k=0,1,2,3$ have not
roots if $s_{0}>1/\beta$. In fact, we have that
\[
W_3(s_0)\left(\beta ^2 s_{0}^2+1\right)^2 = \widetilde{W_3}(s_0),
\]
where
\[
\widetilde{W_3}(s_0)= \dfrac{16 \beta ^3 \left(\beta ^2-1\right)
\left(2 s_{0} \left(s_{0}^2-1\right)+\left(s_{0}^2+1\right)^2
\arccos\left(\dfrac{2}{s_{0}^2+1}-1\right)\right)}{\left(s_{0}^2+1\right)^2}.
\]
Computing the derivative of $\widetilde{W_3}(s_0)$ we have

\begin{equation*}
\dfrac{d\widetilde{W_3}}{d s_0}=\dfrac{256 \beta ^3 \left(\beta
^2-1\right) s_{0}^2}{\left(s_{0}^2+1\right)^3},
\end{equation*}
which is strictly positive for all $s_{0}\neq 0$ and $\beta > 1$, and
strictly negative for all $s_{0}\neq 0$ and $0<\beta < 1$. Therefore,
since $\lim_{s_{0}\rightarrow 0} \widetilde{W_3}(s_0) = 0$ and
$\lim_{s_{0}\rightarrow\infty} \widetilde{W_3}(s_0) = 0$, the
function $\widetilde{W_{3}}$ has no roots for $s_{0} > 0$. Hence, the
function $W_{3}$ has no roots if $s_{0} > 0$ and $\beta\neq 1$.

In summary, the ordered set $\mathcal{F}=[f_{0}, f_{1}, f_{2},f_{3}]$
is an ET-Chebyschev System. By Theorem \ref{teorema-ECT} we conclude
that there exists a linear combination of the functions of
$\mathcal{F}$ with at most three roots. Thus, the upper bound for the
number of zeros of any function in the linear space of functions
generated by the functions of $\mathcal{F}$ is three. In Example
\ref{exemplo-costura} of Section \ref{exemplos} we show that this
upper bound for the zeros is reached.
\end{proof}

\begin{lemma}
The highest amplitude limit cycle, when it exists and is hyperbolic, is stable
$($resp. unstable$)$ provided that
$\xi(b_{11}^-+b_{22}^-)+b_{11}^++b_{22}^+<0$ $($resp.
$\xi(b_{11}^-+b_{22}^-)+b_{11}^++b_{22}^+>0)$. The lowest amplitude
limit cycle, when it exists and is hyperbolic, is asymptotically stable $($resp. unstable$)$
provided that $b_{11}^-+b_{22}^-<0$ or $b_{11}^-+b_{22}^-=0$ and $b
v_1^-+v_1^+>0$ $($resp. $b_{11}^-+b_{22}^->0$ or
$b_{11}^-+b_{22}^-=0$ and $b v_1^-+v_1^+<0)$.
\end{lemma}

\begin{proof}
In what follows we study the stability of the crossing limit cycles.
In fact, this stability depends of the sign of $M_{1}$. Indeed,
$M_{1}(y_0) = M_{1}^-(y_0) - M_{1}^+(y_0)$, let
$y_{0}^{\ast}\in\mathbb{R}^{+}$ such that $M_{1}(y_0) = 0$, if
$M_{1}(y_{0}) > 0 $ for $y_{0}<y_{0}^{\ast}$ and $M_{1}(y_{0}) < 0$
for $y_{0}>y_{0}^{\ast}$, then the crossing limit cycle defined by
$y_0$ is unstable. If these signs are reversed then it is stable.

We get that
\begin{equation*}
\displaystyle\lim_{s_0\rightarrow 0} \widetilde{M}_{1}(s_0)=-2
\pi\alpha  b \xi  (b_{11}^-+b_{22}^-), \,\,\, \mbox{and}\,\,\,
\displaystyle\lim_{s_0\rightarrow \infty}
M_{1}(s_0)=\xi(b_{11}^-+b_{22}^-)+b_{11}^++b_{22}^+.
\end{equation*}
Therefore, since $b<0,\, \alpha > 0$ and $\xi>0$ the
$\sgn\Big(\displaystyle\lim_{s_0\rightarrow 0}
\widetilde{M}_{1}(s_0)\Big)=\sgn(b_{11}^-+b_{22}^-)$, and $2 b \beta
s_{0} \xi ^2 m_{11}^-(s_0)=\widetilde{M}_{1}(s_0)$ so
$\sgn\Big(\displaystyle\lim_{s_0\rightarrow 0}
M_{1}(s_0)\Big)=-\sgn(b_{11}^-+b_{22}^-)$, see Figure
\ref{GraficoRepelerCrossingCycleM1full}. Thus, the lemma follows.
\end{proof}

\begin{figure}[h]
\begin{center}
\begin{overpic}[width=10cm]{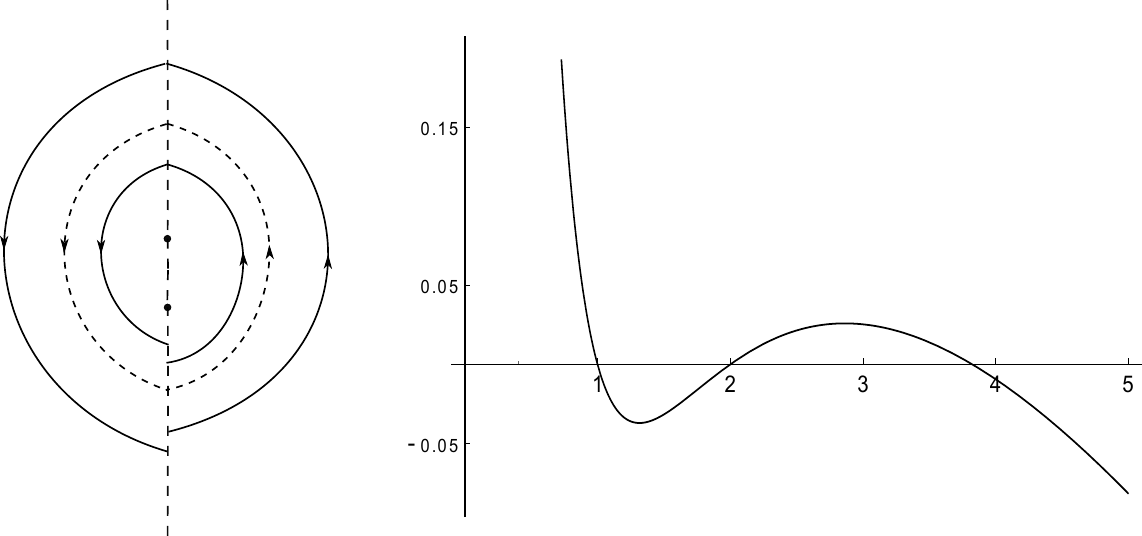}
\put(10,45){$\Sigma$}\put(16,37){$y_0^*$}
\end{overpic}
\end{center}
\caption{The repealer crossing limit cycle and the graphic of
$M_{1}(y_0)$ 1-order linear perturbation of a piecewise linear
center.}\label{GraficoRepelerCrossingCycleM1full}
\end{figure}

Now, we complete the proof of Theorem \ref{t1} analyzing the
stability of the crossing limit cycles and of the periodic orbit at
infinity. In fact, the stability of the periodic orbit at infinity is
given by $\sgn(\xi (b_{11}^-+b_{22}^-) + b_{11}^+  + b_{22}^+)$, see
Section \ref{secao-estudo-infinito}, and the stability of the
crossing limit cycle (c.l.c.) of the biggest amplitude limit cycle is
given in Table \ref{tabela-estab-costura}.

\begin{table}[htb]
\centering
\begin{tabular}{|c|c|c|}
\hline sign$(\xi (b_{11}^-+b_{22}^-) + b_{11}^+  + b_{22}^+)$ &
stability
of the bigger c.l.c & stability of $\infty$ \\
\hline \hline $-1$ & asymptotically stable & unstable  \\ \hline $1$ & unstable &
stable\\ \hline
\end{tabular}
\smallskip
\caption{Stability of the highest amplitude crossing limit cycle and
the infinity.} \label{tabela-estab-costura}
\end{table}

\subsection{Proof of Theorem \ref{t2}}\label{secao-ciclosliding}

In order to study the sliding/escaping limit cycle we consider the
second order linear perturbation of the vector field $Z$ given in
\eqref{p1s}. The singular points of the systems in
$\Sigma^-$ and $\Sigma^+$ are given by $p^{-}=(-e,0)$ and $p^{+}=(-b
d/(a^2+b c),-d/(a^2+b c))$, respectively. The eigenvalues of the
unperturbed piecewise linear vector field $Z_{2, 0}(X)$ in $\Sigma^-$
and $\Sigma^+$ are given by $Spec^{-}=\{i,-i\}$ and
$Spec^{+}=\left\{-\sqrt{a^{2}+b c},\sqrt{a^{2}+b c}\right\}$,
respectively. Consider $Z_{\e}$ given in \eqref{p1s} and assume that
$b_{11}^-=-b_{22}^-$, the fold point in $\Sigma^-$ and $\Sigma^+$ are
given by
\[
\begin{array}{ll}
y_{f1} =    &v_1^-\e + (A_{0}+v_1^-a_{2})\e^{2} +\mathcal{O}(\e^{3})
\, \,  \mbox{and}
\vspace{0.2cm}\\
y_{f2}  =  &-\dfrac{v_1^+\e}{b}+\dfrac{(-b C_{0}+v_1^+
c_{2})\e^{2}}{b^{2}}+\mathcal{O}(\e^{3}),
\end{array}
\]
respectively. Under the assumptions $e>0$ and $db<0$ we get that
$y_{f1}$ is a visible fold point and $y_{f2}$ is an invisible fold
point. Assuming that $v_1^- > -\dfrac{v_1^+}{b}$ we have that
$y_{f1}$ is over of $y_{f2}$ and by \eqref{eq campo filippov} the
expression of the sliding vector fields is
\[
\begin{array}{ll}
Z_{\e}^{s}(0,y) = & y (a y-b e-d) +\e  (v_1^- (d-a y)-y (y (a a_{2}+
b b_{22}^-+b_{22}^+)-a_{2} d+b v_2^-+c_{2} e+d_{0})\vspace{0.2cm}\\
& -v_1^+ e) + \e ^2\Big(-y (a A_{0}-v_1^- b_{22}^+-a_{2} d_{0}-A_{2}
d+b w_2^-+v_2^- c_{2}+b_{22}^- v_1^++C_{2} e \vspace{0.2cm} \\
&+D_{0})  +y^2 (-(a A_{2}-a_{2} b_{22}^++b c_{22}^-+b_{22}^-
c_{2}+D_{2}))+
 v_1^- d_{0}+A_{0} d-v_2^- v_1^+-C_{0} e\Big).
\end{array}
\]
Any point in the sliding region is given as a convex combination of
$y_{f1}$ and $y_{f2}$ as follows
\[
y_s (\lambda) =(1-\lambda)y_{f1} + \lambda y_{f2}=\e \Big(v_1^- -
v_1^- \lambda - \frac{v_1^+ \lambda}{b}\Big),
\]
where $0<\lambda<1$. A necessary condition for the existence of a
sliding/escaping limit cycle is that the sliding vector field is
regular and points toward the visible fold point $y_{f1}$. The
pseudo-equilibrium is  $(0,y^*)$ with $ y^*=(d+be)/a$, which is
reached when $\lambda=\lambda^*=-b(d + b e)/(a(v_1^-b+v_1^+)\e)$.
Under the hypotheses $a<0$ and $d+be > 0$, we obtain that $(0,y^*)
\notin \Sigma^s$, i.e., $Z_{\e}^s$ is regular. The direction of the
sliding vector field is given by the sign of the derivative of $Z_{
\e}^s$ evaluated at $y_s$, that is by $(v_1^- b+v_1^+) (b e+d)/b$.
From assumptions $(d+ b e)>0$ and $(v_1^- b + v_1^+)<0$ ($(v_1^- b +
v_1^+)>0$ resp.), we conclude that the sliding (escaping resp.)
vector field points towards $y_{f1}$ ($y_{f2}$ resp.).

In what follows we study the return maps passing through the fold
point of $Z_{\e}^{+}$ and $Z_{\e}^{-}$. Our goal is to provide an
order relation between the images by the flow of the fold points in a
transverse section through $y_{f1}$. This analysis not only provides
a necessary condition for the existence of a sliding/escaping limit
cycle, but also provides its distinct topological type.  The negative
half return map in a neighborhood of the invisible fold point
$y_{f2}$ defines the involution $\gamma_{Z^+_{\e}}: I^- \rightarrow
I^+$, where $I^{+}, I^-$ is an open interval above, below, resp., of
$y_{f2}$. For more details about the construction of this involution,
see \cite{T1}. In this way, we have that $
\gamma_{Z^+_{\e}}^{-1}(y_{f1})=y_{f3}\in I^-$. The line $\{(x,y);x =
0\}$ is tangent to the fold points. Therefore we cannot use the
Implicit Function Theorem in this case. However, we can obtain a
condition for the existence of a sliding/escaping limit cycle
studying the intersection of the trajectories of $Z_{\e}$ in
$\Sigma^-$ with initial conditions at $y_{f1}, y_{f2}$ and $y_{f3}$
with the line $\Lambda=\{(x,y); x\leq 0, y = y_{f1}\} \subset
\Sigma^-$, which is a transversal section at the fold point $y_{f1}$.
Considering the smooth vector field $Z^-_{\e}$ and the initial
conditions $(0,y_{f1})$, $(0,y_{f2})$ and $(0,y_{f3})$ the
intersection of the flow of $Z^-_{\e}$ with $\Lambda$ define the
return maps $S_0(\e)$, $S_1(\e)$, $S_2(\e)$ and $S_3(\e)$,
respectively, given by
\[
\begin{array}{l}
\begin{array}{ll}
S_0(\e) = &-2 e+\e  (2 b_{21}^- e-2 v_2^-)+\e ^2 \Big(\dfrac{1}{2}
\pi  e (c_{11}^-+c_{22}^-)-2 \Big( v_1^- b_{22}^- - v_2^- b_{21}^- + w_2^- \\
&+ (b_{21}^-)^2 e - c_{21}^- e\Big)\Big)+\mathcal{O}(\e^3),
\end{array}\\
\begin{array}{ll}
S_{1}(\e)  = &-2 e+\e  (2 b_{21}^- e-2 v_2^-)+\e^2 \Big(-2 \Big(v_1^
- b_{22}^- - v_2^- b_{21}^- + w_2^- + (b_{21}^-)^2 e - c_{21}^- e\Big)\\
&-\dfrac{1}{2} \pi  e (c_{11}^-+c_{22}^-)\Big)+\mathcal{O}(\e^3),
\end{array}
\\
\begin{array}{ll}
S_{2}(\e) = & -2 e +\e  (2 b_{21}^- e-2 v_2^-) -\dfrac{\e ^2 }{2 b^2
e} \Big((v_1^-)^2 b^2+2 v_1^- b (2 b b_{22}^- e+v_1^+)+b^2 e \Big(4
\Big(
-v_2^- b_{21}^- \vspace{0.2cm}\\
&+w_2^- + (b_{21}^-)^2 e -c_{21}^- e\Big)-\pi  e
(c_{11}^-+c_{22}^-)\Big)+(v_1^+)^2 \Big)+\mathcal{O}(\e^3),
\end{array}
\\
\begin{array}{ll}
S_{3}(\e) = & -2 e +\e  (2 b_{21}^- e-2 v_2^-) + \dfrac{\e^2 }{2 b^2 e}\Big( -4 b^2 e (v_1^- b_{22}^- - v_2^- b_{21}^-+v_2^-)-4 (v_1^- b+v_1^+)^2\\
&+\pi  b^2 e^2 (c_{11}^- + b_{22}^-) + 4 b^2 e^2 \left(c_{21}^- - (b_{21}^-)^2\right) \Big) +\mathcal{O}(\e^3).
\end{array}
\end{array}
\]
In this way, we conclude:
\begin{itemize}
\item[$(1-a)$]  If $0<c_{11}^-+c_{22}^- < (v_1^- b+v_1^+)^2/(
2 b^2 e^2 \pi)$, then $S_3 < S_2 < S_1 < S_0$ and system $Z_{ \e}$
admits a sliding cycle of Type I.

\item[$(1-b)$] If  $(v_1^- b+v_1^+)^2)/(2 b^2 e^2 \pi)<c_{11}^-
+c_{22}^-<2 (v_1^- b+v_1^+)^2/( b^2 e^2 \pi)$, then $S_3 < S_1 < S_2
< S_0$, and system $Z_{ \e}$ admits a sliding cycle of Type II.
\end{itemize}
\begin{figure}[h]
\begin{center}
\begin{overpic}[width=5cm]{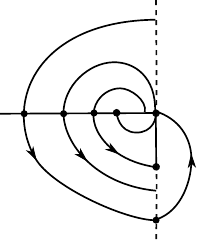}
\put(67,54){$y_{f1}$}\put(67,30){$y_{f2}$}\put(67,5){$y_{f3}$}
\put(60,-2){$\Sigma$}\put(3,55){$S_3$}\put(20,55){$S_1$}\put(33,55)
{$S_2$}\put(43,55){$S_0$}\put(-2,47){$\Lambda$}
\end{overpic}
\end{center}
\caption{The point $y_{f3}$ and the return maps $S_0(\e),S_1(\e),
S_2(\e)$ and $S_3(\e)$.}\label{S1S2S3}
\end{figure}
Working similarly and assuming that $a<0, d+be > 0$ and $v_1^- b +
v_1^+>0$ we can conclude that
\begin{itemize}
\item[$(2-a)$]  If $-(v_1^- b+v_1^+)^2)/(2 b^2 e^2 \pi)
< c_{11}^-+c_{22}^- < 0$, then system $Z_{ \e}$ admits a escaping
cycle of Type I.

\item[$(2-b)$] If $c_{11}^-+c_{22}^- < -(v_1^- b+v_1^+)^2)/
(2 b^2 e^2 \pi)$, then system $Z_{ \e}$ admits a escaping cycle of
Type II.
\end{itemize}
When the assumption $b_{11}^-=-b_{22}^-$ does not hold we see that series around $\e=0$ of $S_0(\e),$ $S_2(\e),$ and $S_3(\e)$ coincide up to order $1$ and writes $-2e + \e \Big(\frac{1}{2} (\pi  e (b_{11}^- + b_{22}^-) - 4 v_2^- + 4 b_{21}^- e)\Big) + \CO(\e^2)$. On the other hand, $S_1(\e)= -2e + \e\Big(\frac{1}{2} (-\pi  e (b_{11}^- + b_{22}^-) - 4 v_2^- + 4 b_{21}^- e) \Big)  + \CO(\e^2)$. In this case, for $\e>0$ sufficiently small, either $S_1(\e)> S_0(\e)$ or $S_1(\e)<S_3(\e)$. Consequently, there is no sliding limit cycle when $b_{11}^-\neq-b_{22}^-$ .

This completes the proof of Theorem \ref{t2}.

\subsection{Proof of Theorem \ref{t3}}\label{secao-simultaniedade}

We divide the proof in two steps. In the first we provide an upper
bound for the number of crossing limit cycles, and in the second we
study the stability of these limit cycles.

\begin{lemma}
Assuming that $b_{11}^-=-b_{22}^-$ the function $M_{1}$, given in
\eqref{M1}, has at most one simple zero. Moreover, there exists
choice of parameters for which this upper bound is reached.
\end{lemma}

\begin{proof}
Consider $M_1(y_0)$ given in \eqref{M1} and assume that
$b_{11}^-=-b_{22}^-$, then we obtain
\begin{equation}\label{M1INF}
M_1(y_0)= 2 v_1^-+ \dfrac{2 v_1^+}{b}+\dfrac{d (b_{11}^+ +
b_{22}^+)}{\xi ^2} -\dfrac{(b_{11}^++b_{22}^+) \left(d^2+\xi ^2
y_{0}^2\right) \arccos\left(\frac{2 d^2}{d^2+\xi ^2
y_{0}^2}-1\right)}{2 y_0 \xi ^3}.
\end{equation}
Proceeding with the change of coordinates $y_0= d s_0/\xi$ and
writing $k_0=2 v_1^-+2 v_1^+/b+d (b_{11}^++b_{22}^+)/\xi ^2$ and
$k_1=-d (b_{11}^++b_{22}^+)/(2 \xi ^2)$ we get
\[
\dfrac{d}{\xi} s_0 M_1 (s_0)= k_{0} s_{0}+ k_{1}
\left(s_{0}^2+1\right) \arccos\left(1-\frac{2}{s_{0}^2+1}\right).
\]
Thus, denoting
\begin{equation*}
\begin{array}{lll}
f_{0}(s_0) & = & s_{0}\\
f_{1}(s_0) & = & \left(s_{0}^2+1\right)
\arccos\left(1-\dfrac{2}{s_{0}^2+1}\right),
\end{array}
\end{equation*}
and computing their Wronskians we obtain
\begin{equation*}
\begin{array}{lll}
W(f_0)(s_0) & = & 1\\ \\
W(f_0,f_1)(s_0) & = & -2 s_{0} + (s_{0}^2-1)
\arccos\left(1-\dfrac{2}{s_{0}^2+1}\right).
\end{array}
\end{equation*}
Let $\widetilde{W}_1(s_0)=W_1/(s_{0}^2-1)$. Therefore
$\dfrac{d\widetilde{W}_1}{ds_0}(s_0)=8
s_{0}^2/(\left(s_{0}^2-1\right)^2 \left(s_{0}^2+1\right))$ which is
strictly positive for all $s_0>0$. Thus, $\widetilde{W}_1$ is strictly
increasing and $W_1$ has at most one zero. The existence of a sliding
limit cycle follows from Theorem \ref{t2}. In Example
\ref{exemplo-sliding} we present a piecewise linear vector field that
exhibits a crossing/sliding limit cycle.
\end{proof}

In the previous case the stability of the crossing limit cycle is
given by the following result.

\begin{lemma}\label{lema-est-ciclo-costura-desl}
The crossing limit cycle of \eqref{p1s} is unstable (resp. stable) if
$v_1^- + \frac{v_1^+}{b} > 0$ $($resp. $v_1^- + \frac{v_1^+}{b} <
0)$.
\end{lemma}

\begin{proof}
The stability of the crossing limit cycle is given by the sign of
$M_{1}$. Indeed, $ M_{1}(y_0) = M_1^-(y_0) - M_1^+(y_0)$, let
$y_{0}^{\ast}\in\mathbb{R}^{+}$ such that $M_{1}(y_0) = 0$,
$M_{1}(y_{0}) > 0 $ for $y_{0}<y_{0}^{\ast}$, and $ M_{1}(y_{0}) < 0
$ for $y_{0}>y_{0}^{\ast}$. Therefore, the crossing limit cycle is
unstable. If $M_{1}(y_{0}) < 0 $ for $y_{0}<y_{0}^{\ast}$ and
$M_{1}(y_{0}) > 0 $ for $y_{0}>y_{0}^{\ast}$, then the crossing limit
cycle is stble. By \eqref{M1INF} we have that
\begin{equation*}
\displaystyle\lim_{y_0\rightarrow 0^+} {M}_{1}(y_0)=2 \left(v_1^- +
\dfrac{v_1^+}{b} \right).
\end{equation*}
Thus, the lemma follows.
\end{proof}


\section{Final remarks and some examples}\label{exemplos}

In this section we provide two examples of piecewise linear vector
fields, the first one admitting three crossing limit cycles and the
second one with a sliding and a crossing limit cycle.

\subsection{Example 1.}\label{exemplo-costura}
Consider the following piecewise linear vector field $Z=(Z^+, Z^-)$
where
\begin{equation*}
\begin{array}{ll}
Z^{-}(x,y)= &\left(\begin{array}{cc}
0 & -1\\
1 & 0
\end{array}\right)\left(\begin{array}{c}
x\\
y
\end{array}\right)
+\left(\begin{array}{c}
0 \\
0.55
\end{array}\right) \\\\
& + \varepsilon\left( \left(\begin{array}{cc}
-1 & 0\\
0 & -1
\end{array}\right)\left(\begin{array}{c}
x\\
y
\end{array}\right)
+\left(\begin{array}{c}
-2.65 \\
0
\end{array}\right)\right)
\end{array}
\end{equation*}
and
\begin{equation*}
\begin{array}{ll}
Z^{+}(x,y)=& \left(\begin{array}{cc}
1 & -1\\
1.01 & -1
\end{array}\right)\left(\begin{array}{c}
x\\
y
\end{array}\right)
+\left(\begin{array}{c}
0 \\
0.1
\end{array}\right) + \varepsilon\left( \left(\begin{array}{cc}
0.21 & 0\\
0 & 0
\end{array}\right)\left(\begin{array}{c}
x\\
y
\end{array}\right)\right).
\end{array}
\end{equation*}
The first order Melnikov function associated to this system is given
by
\begin{equation*}
\begin{array}{ll}
M_1(y_0)=& \dfrac{-1}{400 y_{0}}\Big(-242 \pi -800 \pi  y_{0}^2+
420 \Big(y_{0}^2+1\Big) \arccos\Big(\dfrac{2}{y_{0}^2+1}-1\Big)\\
& +\Big(400 y_{0}^2+121\Big) \arccos\Big(\dfrac{242}{400
y_{0}^2+121}-1\Big)+840 y_{0}\Big),
\end{array}
\end{equation*}
which has three zeros $y_{0}=1, y_{0}=2$ and by the
Newton-Kantorovich method, see \cite{argyros}, we have the third zero
in the neighborhood of $y_{0}=3.82781$, see Figure \ref{GM1}. Each
zero corresponds to a crossing limit cycle with alternating
stability.  In this case, the highest amplitude crossing limit cycle
is unstable and the infinity is stable.
\begin{figure}[h]
\begin{center}
\begin{overpic}[width=8cm]{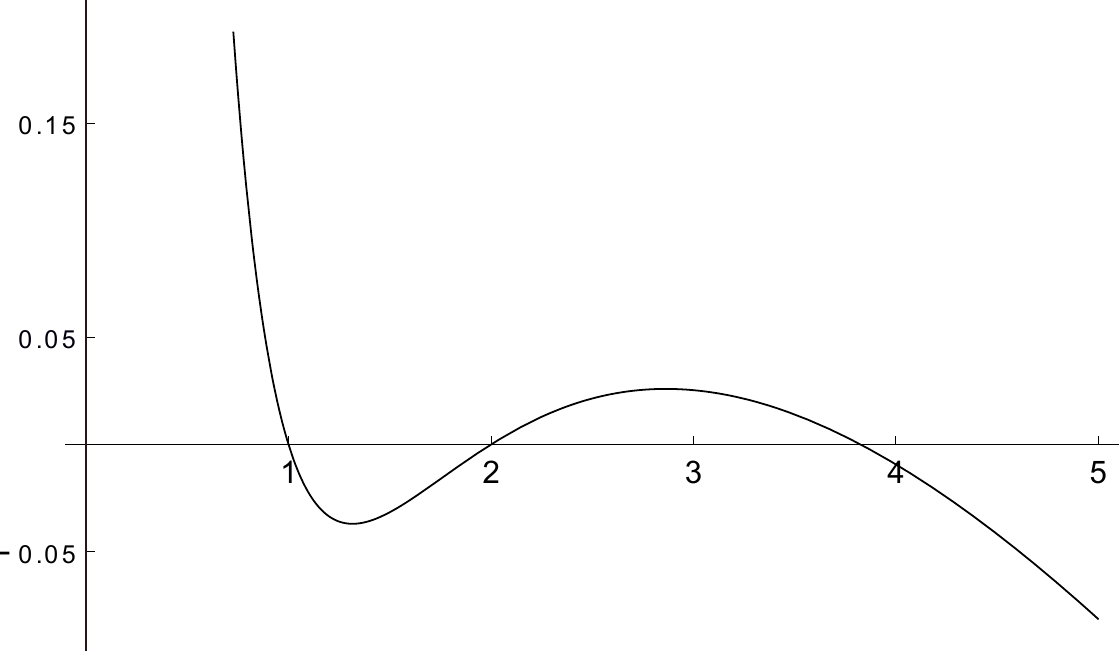}
\end{overpic}
\end{center}
\caption{Graphic of $M_{1}.$ Each zero corresponds to a crossing limit
cycle. }\label{GM1}
\end{figure}

\subsection{Example 2.}\label{exemplo-sliding}
Consider the following piecewise linear vector field $Z=(Z^+, Z^-)$
where
\begin{equation*}
\begin{array}{ll}       Z^{-}(x,y) = &
\left(\begin{array}{cc}
0 & -1\\
1 & 0
\end{array}\right)\left(\begin{array}{c}
x\\
y
\end{array}\right)
        +\left(\begin{array}{c}
            0 \\
            1
        \end{array}\right)
        + \e\left( \left(\begin{array}{cc}
            1 & 0\\
            0 & -1
        \end{array}\right)\left(\begin{array}{c}
            x\\
            y
        \end{array}\right)
        +\left(\begin{array}{c}
            0.2 \\
            0
        \end{array}\right)\right)
        \\\\
&   + \e^{2}\left( \left(\begin{array}{cc}
        0.03 & 0\\
        0 & 0.02
    \end{array}\right)\left(\begin{array}{c}
        x\\
        y
    \end{array}\right)\right),
    \end{array}
    \end{equation*}

    \begin{equation*}
        Z^{+}(x,y)=
        \left(\begin{array}{cc}
            -1 & -1\\
            2  & 1
        \end{array}\right)\left(\begin{array}{c}
            x\\
            y
        \end{array}\right)
        +\left(\begin{array}{c}
            0 \\
            2
        \end{array}\right)
        + \e\left( \left(\begin{array}{cc}
            1.5 & 0\\
            0 & -0.4
        \end{array}\right)\left(\begin{array}{c}
            x\\
            y
        \end{array}\right)
    +\left(\begin{array}{c}
        -0.5 \\
        0
    \end{array}\right)\right).
\end{equation*}
The associated first order Melnikov function is given by
\begin{equation*}
M_{1}(y_0)= 2.2-\left(0.1 y_{0}+\dfrac{1.6}{y_{0}}\right)
\arccos\left(\dfrac{8}{0.25 y_{0}^2+4}-1\right).
\end{equation*}
The graph of $M_{1}$ is given by Figure \ref{GM1S}, the crossing
limit cycle is located in a neighborhood of $y_{0}=7.94622$, is
repealer and the infinite is attractor.
\begin{figure}[h]
\begin{center}
\begin{overpic}[width=7 cm]{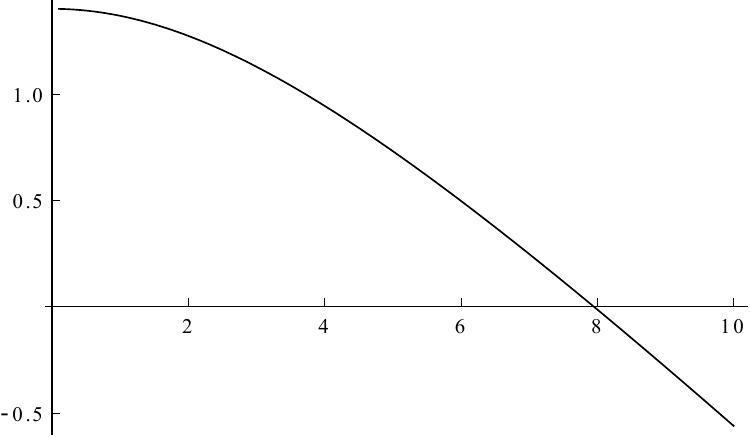}
\end{overpic}
\end{center}
\caption{Graphic of $M_{1}$. The unique zero of $M_1$ corresponds to a
repealer crossing limit cycle.}\label{GM1S}
\end{figure}

\section*{Acknowledgements}
JLC is partially
supported by Goi\'as Research Foundation (FAPEG) and CAPES. JL is partially supported by the Ministerio de Econom\'ia, Industria
y Competitividad, Agencia Estatal de Investigaci\'{o}n grant
MTM2016-77278-P (FEDER), the Ag\`encia de Gesti\'o d'Ajuts
Universitaris i de Recerca grant 2017 SGR 1617, and the European
project Dynamics-H2020-MSCA-RISE-2017-777911. DDN is partially supported by FAPESP grants 2018/16430-8, 2018/ 13481-0, and 2019/10269-3, and by CNPq grant 306649/2018-7 and 438975/ 2018-9.  DJT is partially
supported by PROCAD/CAPES grant 88881.0 68462/2014-01 and by
CNPq-Brazil.

\end{document}